\documentclass[12pt]{amsart}
\usepackage[utf8]{inputenc}

\title[Aluthge transforms of generalized hyperbolic operators]{Shadowing, Generalized hyperbolicity and Aluthge transforms}
\author{Linh T. T. Tran }
\keywords{hyperbolic operator, generalized hyperbolic operator, shifted hyperbolic operator, Aluthge transform, shadowing property, weighted shift}

\subjclass[2020]{Primary: 37C20; Secondary: 
47A15, 47B49, 47B37}

\usepackage{lipsum}

\newcommand\blfootnote[1]{%
  \begingroup
  \renewcommand\thefootnote{}\footnote{#1}%
  \addtocounter{footnote}{-1}%
  \endgroup
}

\usepackage{amssymb,amstext,amsmath,amscd,amsthm,amsfonts,enumerate,latexsym}
\usepackage{color}
\usepackage[all]{xy}
\usepackage{imakeidx}
\usepackage{cases}
\usepackage{graphicx}
\usepackage{enumitem}
\usepackage{indentfirst}
\usepackage[mathscr]{eucal}
\usepackage{graphicx,color}
\numberwithin{equation}{section}

\theoremstyle{plain}
\newtheorem{theorem}{Theorem}[section]

\newtheorem{lemma}[theorem]{Lemma}
\newtheorem{corollary}[theorem]{Corollary}
\newtheorem{question}[theorem]{Question}

\theoremstyle{definition}
\newtheorem*{definition}{Definition}

\theoremstyle{remark}
\newtheorem{remark}[theorem]{Remark}
\usepackage{lineno}
\usepackage[colorlinks]{hyperref}

\begin{document}
%\mainmatter
\begin{abstract}
In this note, we introduce the notion of $r$-homoclinic points. We solve invariant subspace problem (ISP for brevity) for shadowing operators on Banach spaces. Afterwards, we verify that the set of generalized hyperbolic operators is invariant under $\lambda$-Aluthge transforms for every $\lambda \in \left( 0,1 \right)$. Next, the Aluthge iterates of invertible operators converge to hyperbolic operators only if the initial operators are hyperbolic. Finally, we prove that the Aluthge iterates of shifted hyperbolic bilateral weighted shifts diverge and that hyperbolic bilateral weighted shifts with divergent Aluthge iterates exist.
\end{abstract}
\blfootnote{Email: tolinh@o.cnu.ac.kr}
\blfootnote{Department of Mathematics, Chungnam National University, Daehakro 99, Yuseong gu, Daejeon, South Korea (34134).}

\maketitle

\section{Introduction}

\noindent
This work is motivated by three questions related to the linear theory of dynamical systems. The first one deals with homoclinic points and their relationship with hyperbolicity. These points were discovered by Henry Poincar\'e in his thesis about the three-body problem of celestial mechanics \cite{an}.
Every hyperbolic operator has not only the shadowing property but also no nonzero homoclinic points. This follows the below intriguing question.

\begin{question}
\label{q1}
Is a linear operator $T$ hyperbolic if and only if $T$ is shadowing and has no nonzero homoclinic points?
\end{question}

The second question is inspried by a long-standing problem in operator theory named Invariant Subspace Problem (abbr. ISP).

\begin{question}
\label{ISP}
Is it true that every operator $T$ on Hilbert spaces $\mathbb{H}$ or, more generally, Banach spaces $X$ has {\em nontrivial $T$-invariant subspaces}, namely a closed subspace $M\notin \left \{ \{0\},X \right \}$ such that $T(M)\subset M$.
\end{question}

 This problem has been intensively studied especially for Hilbert spaces, and positive results have been obtained for many classes of operators.
The answer, however, is negative in general Banach spaces that was shown by Enflo \cite{en} and still open for separable Hilbert spaces. Furthermore, it turns out to relate with the recent question if a shadowing operator is generalized hyperbolic \cite{ddm}. \\
\indent
Generalized hyperbolicity is a new concept in linear dynamics which generalizes hyperbolicity. In fact, every generalized hyperbolic operator is either a uniform contraction or expansion or satisfies ISP following the definition. Hence, it is worth asking whether one can classify shadowing operators in a similar way.

\begin{question}
\label{question shadowing classification}
Is every shadowing operator on Banach space either a uniform contraction or expansion or satisfies ISP?
\end{question}

Our last concern focuses on the dynamics of the Aluthge transforms for linear operators on Hilbert spaces which have attracted the attention of Operator Theorists during the recent 30 years.
The following question was stated by Cirilo et al. \cite{Cirilo1}.

\begin{question} \cite{Cirilo1}
\label{main question}
Does the Aluthge transformation keep the space of generalized hyperbolic operator invariant? Is it true that the sequence of iterates of a shifted hyperbolic operator by the Aluthge transformation does not converge?
\end{question}
In this paper, we will give partial positive answers for above questions.
Regarding to Question \ref{q1}, we introduce the notion of $r$-homoclinic points  and prove that a linear operator is hyperbolic if and only if it is shadowing and has no nonzero $r$-homoclinic points. This result is shown in detail in Lemma \ref{lemma Hr(T) and Ec(T)}, Lemma \ref{lemma T hyperbolic iff EcT closed}, and Theorem \ref{theorem characterize hyperbolicity}.

As regards Question \ref{question shadowing classification}, Theorem \ref{theorem classification shadowing} proves that a shadowing operator on a Banach space $X$ whose set of bounded orbits is not dense, is either a uniform contraction or an expansion or it has nontrivial closed $T$-invariant subspaces.\\

Related to Questions \ref{ISP} and \ref{main question}, 
we show that the sets of generalized hyperbolic operators and shifted hyperbolic operators are invariant under the $\lambda$-Aluthge transforms for every $\lambda \in \left( 0,1 \right)$ by using Theorem \ref{proposition lambda Aluthge transform of GH} and Corollary \ref{Cor lambda Aluthge of SH}.\\

Next, in Theorem \ref{prop L is hyperbolic then T is hyperbolic}, we verify that the Aluthge iterates of an invertible operator converges to a hyperbolic operator only if the initial operator is hyperbolic. Finally, we initiate a discussion about the divergence of Aluthge iterates corresponding to generalized hyperbolic weighted shifts in Theorem \ref{Aluthge iterates of SH weighted shifts is divergent} and Theorem \ref{Theorem Aluthge iterates of hyperbolic operator}. 

\section{Statement of the results}
\label{statemens of the results}
\noindent
Without further mention, we assume that $T$ is an automorphism on a Banach space $X$ through out this article. Hyperbolic operators are determined in \cite{Bernades1} as below.
\begin{definition}
\label{hyperbolic operators-OT}
An invertible operator $T$ on Banach space $X$ is said to be \textit{ hyperbolic} if $\sigma \left( T \right) \cap \mathbb{T} = \varnothing$, where $\mathbb{T}$ is the unit circle of complex plane $\mathbb{C}$.
\end{definition}
Let us denote by $r \left( .\right)$ the spectral radius of the respective operator. We recall that an operator $T:X\to X$ on Banach space $X$ is a {\em uniform contraction}
if its spectral radius $r(T)<1$. Likewise, if $T$ is invertible and $T^{-1}$ is uniformly contract then $T$ is said to be a {\em uniform expansion}. 
An equivalent definition of hyperbolic operator is stated below.
\begin{remark}
\label{remark hyperbolic operators}
On a Banach space $X$, $T \in B \left( X \right)$ is hyperbolic if and only if there is a splitting $X =M \oplus N$, $T = T_M \oplus T_N$, where $M$ and $N$ are closed $T$-invariant subspaces of $X$, $T_M = T \mid M$ and $T_N = T \mid N$, satisfying the following.
\begin{enumerate}%[label=(\roman*)]
 \item $T \left( M \right) = M$, $T^{-1} \left( N \right) = N$;
 \item $T _M$ and $T^{-1}_N $ are uniform contractions.
\end{enumerate}
\end{remark}

By replacing the equalities by inclusions on item (1) above we obtain the following concept (see Cirilo et al \cite{Cirilo1}, Bernades and Messaoudi \cite{ddm}, and Bernades et al \cite{Bernades1}).

\begin{definition}
\label{def generalized hyperbolic operators}
An automorphism $T$ on a Banach space $X$ is {\em generalized hyperbolic} if there exists a decomposition of $X$ in complementary closed subspaces,
$X =M \oplus N$, satisfying
\begin{enumerate}%[label=(\roman*)]
 \item $T \left( M \right) \subset M$ and $T^{-1} \left( N \right) \subset N$;
 \item $T_M$ and $T^{-1}_N$ are uniform contractions.
\end{enumerate}
\end{definition}

A generalized hyperbolic operator is said to be {\em shifted hyperbolic} if $T^{-1}(M)\cap N$ is nontrivial. Equivalently, $T$ is a shifted hyperbolic operator if and only if it is generalized hyperbolic but not hyperbolic (see Theorem 2 in \cite{Cirilo1}).

On the other hand, an invertible operator $T: X \longrightarrow X$ has the shadowing property (or $T$ is {\em shadowing}) if for every $\epsilon >0$, there is $ \delta >0$ such that for every sequence $\left \{ x_n \right \}_{n \in \mathbb{Z}}$ with $\left \| T \left( x_n \right) - x_{n+1} \right \| < \delta$ for every $n\in\mathbb{Z}$, there exists $ x \in X$ satisfying $\left \| T^{n} \left( x \right) - x_n \right \| < \epsilon$ for every $ n \in \mathbb{Z}$ (see \cite{Bernades2}).

A point $x \in X$ is called a {\em homoclinic point} (associated with $0$) of $T$ if
  
\begin{align}
    \lim_{n \rightarrow \pm \infty}T^{n} \left( x \right) =0 \label{eq1 homoclinic}.
\end{align}

It follows from Theorem 2 in \cite{Cirilo1} that a linear operator is hyperbolic if and only if it is generalized hyperbolic and has no nonzero homoclinic points. The question whether an invertible operator has the shadowing property if and only if it is generalized hyperbolic was posed in \cite{ddm}. Thus, Question \ref{q1} is derived from these statements by replacing generalized hyperbolicity by shadowing. We give a partial positive answer for this question based on the following definition.

%%%%%%%%%%%%%%%%%%%%%%55
%%%%%%%%%%%%%555
%%%%%%%%%%%%%%%%%55
%%%%%%%%%%%%%%% CHECK DEN DAY %%%%%%%%%%%%%5

\begin{definition}
\label{def epsilon homoclinic points}
Let $r>0$. We say that $x \in X$ is an {\em $r$-homoclinic point}
if there is $N\in\mathbb{N}$ such that
$\|T^n(x)\|\le r$ for every $n\in\mathbb{Z}$ with
$|n|\geq N$.
\end{definition}
In general, if we denote by $H \left( T \right)$ (resp. $H_{r} \left( T \right)$) the set of homoclinic points associated to $0$ (resp. $r$-homoclinic points) of $T$ then
  
\begin{align*}
    H \left( T \right) = \cap _{r>0} H _{r} \left( T \right).
\end{align*}
  
We have the following remark.
\begin{remark}
\label{remark epsilon homoclinic point}
The following properties are equivalent.
\begin{enumerate}[label=(\roman*)]
\item $\exists r >0$ such that $H_{r}\left( T \right) = \left \{ 0 \right \}$.
\item $H_{r}\left( T \right) = \left \{ 0 \right\}$ for every $r > 0$.
\end{enumerate}
\end{remark}
\begin{proof}
Obviously, item (ii) follows item (i). So, we only need to prove that item (i) follows item (ii). \\
\indent
Let us assume that there exists $r>0$ satisfying (i), $r'>0$ arbitrary and $x \in H_{r'} \left( T \right)$. By definition, there exists $N' \in \mathbb{N}$ associated with $r'$ so that
  
\begin{align*}
    \left \| T^{n} \left( \frac{r}{r'} x \right) \right \| = \frac{r}{r'} \left \| T^{n} \left( x \right) \right \| \le \frac{r}{r'}r' = r, \quad \forall n \in \mathbb{N}, \left| n \right| \ge N'.
\end{align*}
  
\indent Therefore, $\frac{r}{r'}x \in H_{r}\left( T \right) = \left \{ 0 \right \}$, implying $x =0$. Since $r'$ is chosen randomly, item (ii) holds.
\end{proof}
For an arbitrary invertible $T : X \longrightarrow X$ on Banach space $X$, we define the set of points with bounded orbit associated with $T$ by
  
\begin{align*}
    E^{c} \left( T \right) = \left \{ x \in X \mid \sup_{n \in \mathbb{Z}} \left \| T^{n} \left( x \right) \right \| < \infty \right \}. 
\end{align*}
  
\indent 
As a consequence, we characterize hyperbolicity of invertible operator $T \in B \left( X \right)$ by $H_r\left( T \right)$ as our first result below.
\begin{theorem}
\label{theorem characterize hyperbolicity}
An invertible operator $T: X \longrightarrow X$ on Banach space $X$ is hyperbolic if and only if $T$ is shadowing and $H_r\left( T \right)$ is closed for every $r >0$.
\end{theorem}

\iffalse
 \begin{theorem}
\label{cor hyperbolic and epsilon homoclinic}

\end{theorem}
\fi

We now provide a partial positive answer for Question \ref{question shadowing classification} as follows.
\begin{theorem}
\label{theorem classification shadowing}
If $T$ is an invertible shadowing operator on a Banach space $X$ then one of the following alternatives holds for $T$. 
\begin{enumerate}[label=(\roman*)]
    \item $T$ is a uniform contraction;
    \item $T$ is a uniform expansion;
    \item $E^{c} \left( T \right)$ is dense;
    \item $T$ satisfies ISP.
\end{enumerate}
\end{theorem}

Now we focus on Hilbert spaces $\mathbb{H}$. The Aluthge transform $\Delta(T)$ of a bounded linear operator $T:\mathbb{H}\to \mathbb{H}$,  which has drawn the attention in Operator Theory during the last 30 years, is defined by
  
\begin{align*}
    \Delta \left( T \right) = \left| T \right| ^{\frac{1}{2}} U \left| T \right| ^{\frac{1}{2}},
\end{align*}
  
where $U$ is the unique polar decomposition of
$T = U \left| T \right|$,
$\left| T \right| = \left( T^*T \right)^{\frac{1}{2}}$ and $\ker \left( \left| T \right| \right) = \ker \left( T \right) = \ker \left( U \right)$. I. B. Jung \cite[Theorem 2.1]{IBJ2} shows that
  
\begin{align*}
    \sigma \left( T \right) = \sigma \left( \Delta \left( T \right) \right)
\end{align*}
  
where $\sigma \left( . \right)$ are the corresponding spectra. Moreover, Jung-Ko-Pearcy \cite[Corollary 1.16]{IBJ1} proves that $T$ has nontrivial closed invariant subspaces if and only if $\Delta \left( T \right)$ does.\\
\indent Another reason relates with the Aluthge iterates determined by
  
\begin{align*}
    \begin{cases}
\Delta ^{(0)} \left( T \right) = T \\ 
\Delta ^{(n)} \left( T \right) = \Delta \left( \Delta ^{(n-1)} \left( T \right) \right), \quad\quad\forall n \ge 1,  n \in \mathbb{Z}^{+}.
\end{cases} 
\end{align*}
  
\indent 
The Aluthge iterates of bounded linear operators are convergent to normal operators as $\dim \left( \mathbb{H} \right)<\infty$ that has been verified in \cite[Theorem 4.2]{Antezana1} but not in general as shown in \cite[Corollary 3.2]{IBJ2}.\\
\indent 
The following statement is a generalization of Aluthge transforms which are named $\lambda$-Aluthge transforms in \cite{Okubo1} and researched in \cite{Antezana2, Huang1, Antezana3}.
\begin{definition} 
\label{def lambda Aluthge transform}
Let $T = U \left| T \right| \in B \left( \mathbb{H} \right)$ be the polar decomposition of $T$ in which $\mathbb{H}$ is a Hilbert space. For any number $\lambda \in \left( 0,1 \right)$, the { \it $\lambda$-Aluthge transform} of $T$ is
  
\begin{align*}
    \Delta_{\lambda} \left( T \right) = \left| T \right|^{\lambda} U \left| T \right|^{1-\lambda}.
 \end{align*}
  
Similar to Aluthge iterates, the $\lambda$-Aluthge iterates of $T$ associated with $\lambda \in \left( 0,1 \right)$ is determined as the following.
  
\begin{align*}
\begin{cases}
    \Delta ^{(0)} _{\lambda} \left( T \right) = T\\
    \Delta ^{(n)} _\lambda \left( T \right) = \Delta _{\lambda} \left( \Delta _{\lambda} ^{(n-1)} \left( T \right) \right), \quad \forall \lambda \in \left( 0,1 \right), n \in \mathbb{Z}^{+}.
\end{cases}
\end{align*}
  
Each element $\Delta _{\lambda} ^{(n)} \left( T \right)$ is called an $n$-th $\lambda$-Aluthge transform of $T$.
\end{definition}

Our below results show partial positive answers for Question \ref{main question}.

\begin{theorem}
\label{proposition lambda Aluthge transform of GH}
For an arbitrary number $0 < \lambda < 1$, the generalized hyperbolicity of bounded linear operators on Hilbert spaces is invariant under $\lambda$-Aluthge transforms.
In particular, for $\lambda = \frac{1}{2}$, T is a generalized hyperbolic operator if and only if its Aluthge transform $\Delta \left( T \right)$ is.
\end{theorem}

 \begin{theorem}
\label{prop L is hyperbolic then T is hyperbolic}
Let $T$ be an invertible operator on Hilbert space $\mathbb{H}$. If the Aluthge iterates $\Delta^{(n)} \left( T \right)$ of $T$ converges to an invertible hyperbolic operator then $T$ is hyperbolic.
\end{theorem}

\begin{theorem}
\label{Aluthge iterates of SH weighted shifts is divergent}
The Aluthge iterates of shifted hyperbolic bilateral weighted shifts on $l^2 \left( \mathbb{Z} \right)$ diverge.
\end{theorem}

\begin{theorem}
\label{Theorem Aluthge iterates of hyperbolic operator}
There exists a hyperbolic operator on some Hilbert spaces whose Aluthge iterates diverge.
\end{theorem}
We now verify aforementioned results in detail.
\section{Proof of the theorems}
\noindent
The following lemma exhibits the relationship between $H_r \left( T \right)$ and $E^{c} \left( T \right)$, which is applied for verifying Theorems \ref{theorem characterize hyperbolicity} and \ref{theorem classification shadowing}.
\begin{lemma}
\label{lemma Hr(T) and Ec(T)}
If $T : X \longrightarrow X$ has the shadowing property on Banach space $X$ then
  
\begin{align}
    H_r \left( T \right) \subseteq E^{c} \left( T \right) \subseteq \overline{H_r\left( T \right)}, \quad \left( \forall r>0 \right)\label{eq inclusion} .
\end{align}
  
\end{lemma}
\begin{proof}
By the definition of $H_r\left( T \right)$, $H_r \left( T \right) \subseteq E^{c} \left( T \right)$ (for every $r >0$) trivially. Therefore, it is necessary to show that $E^{c} \left( T \right) \subseteq \overline{H_{r} \left( T \right)}$ for every $ r >0$.\\
Since $\epsilon$ can be chosen arbitrarily, without loss of generality, let us take $r>0$ such that $0< \epsilon <r$  and $x \in E^{c}\left( T \right)$. Since $T$ has shadowing property, there exists $\delta >0$ corresponding to $\frac{\epsilon}{2}$ which makes $T$ become $\frac{\epsilon}{2}$-shadowing. We choose $0<M< \infty$ such that
  
\begin{align*}
    \sup_{n \in \mathbb{Z}} \left \| T^{n} \left( x \right) \right \| <M.
\end{align*}
  
Besides, let $\left \{\beta_n \right \}_{n \in \mathbb{Z}} \subseteq \left( 0,1 \right]$ be a sequence defined by
  
\begin{align*}
    \begin{cases}
\beta_{0}=1\\
\left| \beta_{n}-\beta_{n+1} \right| < \frac{\delta}{2M} \quad \left( \forall n \in \mathbb{Z} \right).\\
\beta_n \longrightarrow 0 \text{ as }n \longrightarrow \pm \infty.
\end{cases}
\end{align*}
  
We define sequence $\left \{ x_n \right \}_{n \in \mathbb{Z}} \subseteq X$ by
  
\begin{align*}
    x_n = \beta_n T^{n} \left( x \right), \quad \forall n \in \mathbb{Z}.
\end{align*}
  
It leads to
  
\begin{align*}
    \left \| T \left( x_n \right) -x_{n+1} \right \| &= \left \| T \left( \beta_{n} T^{n} \left(x \right) \right) - \beta_{n+1} T^{n+1} \left( x \right) \right \|\\
    &= \left \| \beta_n T^{n+1} \left( x \right) - \beta_{n+1} T^{n+1} \left( x \right) \right \|\\
    &= \left | \beta_{n}-\beta_{n+1} \right | \left \| T^{n+1} \left( x \right) \right \| \\
    & < \frac{\delta}{2M}.M =\frac{\delta}{2} < \delta, \quad \forall n \in \mathbb{Z}.
\end{align*}
  
It turns out that $\left \{ x_n \right \}_{n \in \mathbb{Z}}$ is a $\delta$-pseudotrajectory of $T$. Hence, there is $y\in X$ so that
  
\begin{align*}
    \left \| T^{n} \left( y \right) - x_n \right \| \le \frac{\epsilon}{2}, \quad \forall n \in \mathbb{Z}
\end{align*}
  
because $T$ is shadowing. On the other side,
  
\begin{align*}
    \left \| T^{n} \left( y \right) \right \| &= \left\| T^{n} \left( y \right) - x_n + x_n \right\| \\
    & \le \left \| T^{n} \left( y \right) - x_n \right \| + \left \| x_n \right \|\\
    & \le \frac{\epsilon}{2} + \beta_{n} \left \| T^{n} \left( x \right) \right \|\\
    &\le \frac{\epsilon}{2} + \beta_n M \longrightarrow \frac{\epsilon}{2}
\text{ as } n \longrightarrow \infty.
\end{align*}
  
It follows the existence of $N \in \mathbb{N}$ so that
  
\begin{align*}
    \left \| T^{n} \left( y \right) \right\| < \epsilon < r, \quad \forall n \in \mathbb{Z}, \left| n \right| \ge N.
\end{align*}
  
In other words, $y \in H_{r} \left( T \right)$. Moreover,
  
\begin{align*}
    \left \| x-y \right \| = \left \| x_0 - T^{0} \left( y \right) \right \| \le \frac{\epsilon}{2} < \epsilon.
\end{align*}
  
So, $x \in \overline{ H_{r} \left( T \right)}$ and we obtain the desired result.
\end{proof}

It is easy to see that $E^{c} \left( T \right)$ is a $T$-invariant subspace. In the following lemma, we verify that $\overline{H_r\left( T \right)}$ is a $T$-invariant subspace of $X$ as equation \eqref{eq inclusion} holds.

\begin{lemma}
\label{lemma Hr(T) is invariant subspace}
Let $T: X \longrightarrow X$ be an invertible linear operator on Banach space $X$ so that equation \eqref{eq inclusion} holds. Then $\overline{H_r\left( T \right)}$ is a closed $T$-invariant subspace of $X$ for every $r>0$.
\end{lemma}
\begin{proof}
By definition, $H_r\left( T \right)$ is a $T$-invariant subset. Besides, $E^{c} \left( T \right)$ is a subspace of $X$ and $0 \in H_{r}\left( T \right)$. Thus $0 \in \overline{H_{r}\left( T \right)}$.\\
We now fix $x$, $y \in \overline{H_r \left( T \right)}$ and $\lambda \in \mathbb{C}$ then there exist sequences $x_n$, $y_n \in H_{r} \left( T \right)$ with
  
\begin{align*}
    \begin{cases}
x_n \longrightarrow x\\
y_n \longrightarrow y
\end{cases} \text{as } n \longrightarrow \infty.
\end{align*}
  
Since $H_{r} \left( T \right) \subseteq E^{c}\left( T \right)$ in equation \eqref{eq inclusion},
  
\begin{align*}
    x_n + \lambda y_n \in E ^{c} \left( T \right).
\end{align*}
  
Furthermore,
  
\begin{align*}
    x + \lambda y = \lim_{n \longrightarrow \infty} \left( x_n + \lambda y_n \right) \in \overline{E^{c} \left( T \right)} = \overline{H_r\left( T \right)}.
\end{align*}
  
Finally, $\overline{H_r\left( T \right)}$ is now a $T$-invariant subspace of $X$.
\end{proof}
The below result was proved in \cite[Theorem 3]{lm}.
\begin{lemma}\cite[Theorem 3]{lm}
\label{lemma T hyperbolic iff EcT closed}
A linear operator $T:X \longrightarrow X$ on Banach space $X$ is hyperbolic if and only if $T$ has shadowing property and $E^c \left( T \right)$ is closed.
\end{lemma}

\begin{proof}[Proof of Theorem \ref{theorem characterize hyperbolicity}]
Firstly, if $T$ is hyperbolic, $T$ has shadowing property and $E^{c}\left( T \right) = \left \{ 0 \right \}$ closed (see \cite[Page 3 and Proposition 19]{Bernades1}). Therefore, $H_r\left( T \right) = \left \{ 0 \right \}$ by \eqref{eq inclusion} for every $r>0$.\\
\indent Conversely, we assume that $T$ has shadowing property and $H_r\left( T \right)$ is closed for some $r>0$. Then $E^{c} \left( T \right) \subseteq \overline{ H _r \left( T \right) }$ by \eqref{eq inclusion} and $E^{c} \left( T \right)$ is closed. Applying Lemma \ref{lemma T hyperbolic iff EcT closed}, $T$ is a hyperbolic operator. As a direct consequence, invertible shadowing operator $T$ on Banach space $X$ is hyperbolic if and only if it has no nonzero $r$-homoclinic points for every $ r>0$.
\end{proof}

\iffalse \begin{proof}[Proof of
Theorem \ref{cor hyperbolic and epsilon homoclinic}] Theorem \ref{cor hyperbolic and epsilon homoclinic} is a direct consequence of  Theorem \ref{theorem characterize hyperbolicity}.
\end{proof}
\fi

\begin{proof}[Proof of Theorem \ref{theorem classification shadowing}]
From Lemma \ref{lemma Hr(T) is invariant subspace} and shadowing property of $T$, $\overline{H_r \left( T \right)}$ is a closed $T$-invariant subspace of $X$ for all $ r >0$. We now assume that $T$ does not satisfy ISP then $\overline{H_r \left( T \right)}$ is either $\left \{ 0 \right \}$ or $X$.\\
\indent First, if $\overline{H_r\left( T \right) = \left \{ 0 \right \}}$ then $E^{c} \left( T \right) \subseteq \overline{H_r\left( T \right)} = \left \{ 0 \right \}$ . Therefore, $T$ is hyperbolic by Lemma \ref{lemma T hyperbolic iff EcT closed}. That means, $T$ is either a uniform contraction (for $r\left( T \right) <1$) or a uniform expansion (for $r \left( T \right) >1$). On the other side, $H_r\left( T \right) \subseteq E^{c} \left( T \right)\subseteq  \overline{H_r\left( T\right)}=X$ implies that $E^{c} \left( T \right)$ is dense.\\
\indent Totally, if $T$ is a shadowing operator on Banach space $X$ that does not satisfy ISP then one of the alternatives (i), (ii) or (iii) must hold for $T$. Otherwise, $T$ satisfies ISP in item (iv).
\end{proof}

It is well-known that every generalized hyperbolic operator is shadowing (see \cite[Theorem A]{Bernades1}). In addition, hyperbolicity and shadowing are preserved under linear conjugacy. We will prove a similar result for generalized hyperbolicity in next lemma.
\begin{lemma}
\label{main lemma}
Let $L \in B \left( X \right)$ be a generalized hyperbolic operator in which $X$ is a Banach space. For every invertible operator $H: X \longrightarrow X$, $HLH^{-1}$ is generalized hyperbolic.
\end{lemma}

\begin{proof}
By definition, there  exists a splitting of subspaces of $X$ w.r.t $L$, which is
$X = M \oplus N$,
where $M$ and $N$ are, respectively, the stable and the unstable subspace of $X$ corresponding to $L$.

Since $H$ is invertible in the assumption, there is a subspace decomposition of the image $H \left( X \right) =X$, which is
  
\begin{align*}
    X = \widehat{M} \oplus \widehat{N}
\end{align*}
  
where $\widehat{M} = H \left( M \right)$ and $\widehat{N} = H \left( N \right)$.\\
\indent
Obviously, $\widehat{M}$ and $\widehat{N}$ are closed subspaces of $X$ because $M$ and $N$ are closed subspaces of $X$ and $H$ is an invertible operator. Consequently, we obtain
  
 \begin{align*}
     \left( HLH^{-1} \right) \left( \widehat{M} \right) &= \left( HL \right) \left( H^{-1} \left( \widehat{M} \right) \right)\\
     & = HL \left( M \right) \quad \left( \text{because } H^{-1} \left( \widehat{M} \right) = H^{-1} \left( H \left(M \right) \right) = M\right)\\
     & \subseteq H \left( M \right) \quad \left( \text{because } L  \text{ is generalized hyperbolic}   \right)\\
     &= \widehat{M}.
 \end{align*}
   
\indent Hence, $\left( HLH^{-1}\right) \left( \widehat{M}\right) \subseteq \widehat{M}$. Similarly,
   
 \begin{align*}
  \left( HL^{-1} H^{-1} \right) \left( \widehat{N} \right) &= HL^{-1} \left( H^{-1} \left( \widehat{N} \right) \right)\\
  &= HL^{-1} \left( H^{-1} \left( H \left( N \right) \right) \right)\\
  &= HL^{-1} \left( N \right)\\
  & \subseteq H \left( N \right) \quad \left( \text{because } L \text{ is generalized hyperbolic}\right) \\
  &= \widehat{N} .
\end{align*}
  
Therefore, $\left( HL^{-1} H^{-1} \right)\left( \widehat{N} \right) \subseteq \widehat{N}$.\\
\indent 
Let us show that $r_1 = r \left( HLH^{-1} \mid \widehat{M} \right)$ and $r_2 = r \left( HL^{-1} H^{-1} \mid \widehat{N} \right)$ are less than $1$. To complete these tasks, for an arbitrary positive integer $n \in \mathbb{N}$, we have the following observations.
  
\begin{align*}
    & \left( HLH^{-1} \right) ^n = \left( HLH^{-1} \right) \dots \left( HLH^{-1} \right) = H L^nH^{-1};\\
    &\left( HL^{-1} H^{-1}\right) ^n = \left( HL^{-1} H^{-1} \right) \dots \left( HL^{-1} H^{-1} \right) = H \left( L^{-1} \right) ^n H^{-1}.
 \end{align*}
   
 Based on Gelfand's formula for calculating the spectral radius,
   
 \begin{align}
     r_1 &= r \left( HLH^{-1} \mid \widehat{M} \right) \\
     &= \lim _{n \rightarrow \infty} \left \| \left( HLH^{-1} \mid \widehat{M} \right) ^n \right \| ^{\frac{1}{n}} \notag \\
     &= \lim _{n \rightarrow \infty} \left \| HL^nH^{-1} \mid \widehat{M} \right \| ^{\frac{1}{n}}. \label{eq 3}
 \end{align}
   
 Since $H$ is an invertible operator from $X$ onto itself, by Bounded Inverse Theorem, $H^{-1}$ is also invertible and bounded.\\
\indent If we take an arbitrary vector $\widehat{v} \in \widehat{M} = H \left( M \right)$ then there exists a vector $v\in M$ such that $v = H^{-1} \left( \widehat{v} \right)$ because $H$ and $H^{-1}$ are invertible operators. So,
   
 \begin{align*}
    \left( HL^nH^{-1}\mid \widehat{M} \right) \left( \widehat{v} \right) &= HL^n \left( H^{-1} \left( \widehat{v} \right) \right) \quad \left( \text{ by definition} \right)\\
    & = HL^n \left( v \right) \quad \left( \text{because } H^{-1}  \left( \widehat{v}\right) =v \right)\\
    &= \left( HL^n \mid M  \right) \left( v \right).
 \end{align*}
   
 As a result,
   
 \begin{align*}
     \left \| \left( HL^nH^{-1} \mid \widehat{M } \right) \left( \widehat{v} \right) \right \| &= \left\| \left( HL^n \mid M\right) \left( v \right) \right \|\\
     & \le \left \| HL^n \mid M \right \| \left \| v \right \| \\
     &=  \left \| HL^n \mid M \right \| \left \| H^{-1} \left( \widehat{v} \right) \right \|\\
     & \le \left \| \left( HL^n \mid M \right)\right \| \left \| H^{-1} \right \| \left \| \widehat{v} \right \| .
 \end{align*}
   
 By definition,
   
 \begin{align*}
     \left \| HL^nH^{-1} \mid \widehat{M} \right \| &= \sup_{\left \| \widehat{v} \right \| =1, \widehat{v} \in \widehat{M}} \left \| \left( HL^nH^{-1} \mid \widehat{M} \right) \left( \widehat{v} \right) \right \| \\
     &\le \sup_{\left \| \widehat{v} \right \| =1, \widehat{v} \in \widehat{M}} \left \| HL^n \mid M \right \| \left\| H^{-1}
 \right \| \left \| \widehat{v}  \right \|.
 \end{align*}
   
 It follows that
   
 \begin{align}
     \left \| HL^nH^{-1} \mid\widehat{M} \right \| \le \left \| HL^n \mid M \right \| \left \| H^{-1} \right \| \label{eq 4}. 
 \end{align}
   
Substituting \eqref{eq 4} to \eqref{eq 3}, we get
  
\begin{align*}
    r_1 &= \lim _{n \rightarrow \infty} \left \| HL^nH^{-1} \mid \widehat{M} \right \| ^{\frac{1}{n}}\\
    & \le \lim _{n \rightarrow \infty} \left( \left \| HL^n \mid M \right \| \left \| H^{-1} \right \| \right) ^{\frac{1}{n}}\\
    & \le \lim _{n \rightarrow \infty} \left( \left \| H \right \| \left \| L^n \mid M \right \| \left \| H^{-1} \right \| \right) ^{\frac{1}{n}} \\
    &= \lim _{n \rightarrow \infty} \left( \left \| H \right \| ^{\frac{1}{n}} \left \| H^{-1} \right \| ^{\frac{1}{n}} \left \| L^n \mid M \right \| ^{\frac{1}{n}} \right)\\
    &= \lim_{n \rightarrow \infty} \left( \left \| L^n \mid M \right \| \right) ^{\frac{1}{n}} \\
    &= r \left( L \mid M \right) <1.
\end{align*}
  
The inequality $r_2 = r \left( HL^{-1} H^{-1} \mid \widehat{N} \right) <1$ is implied by repeating this process.\\
\indent Finally, according to the second definition in Section \ref{statemens of the results}, $HLH^{-1}$ is generalized hyperbolic.
\end{proof}

The space of generalized hyperbolic operators $T \in B \left( \mathbb{H} \right)$ on Hilbert space $\mathbb{H}$ is invariant under not only Aluthge transforms but also $\lambda$-Aluthge transforms. In light of Lemma \ref{main lemma}, we obtain our second result below.
\begin{proof}[Proof of Theorem \ref{proposition lambda Aluthge transform of GH}]
Since $T$ is invertible, $\left| T \right| ^{\lambda}$ is an invertible operator for every $\lambda \in \left( 0,1 \right)$. Besides, $T = U\left| T \right| $ implies that
  
\begin{align}
\label{eq calculate U}
    T \left| T \right| ^{-1}= U \left| T \right| \left| T \right| ^{-1} =U.
\end{align}
  
Substituting \eqref{eq calculate U} to the Aluthge transform of $T$, we obtain
  
\begin{align*}
    \Delta_{\lambda} \left( T \right) &= \left| T \right| ^{\lambda} U \left| T \right| ^{1-\lambda} \quad \left( \forall \lambda \in \left(0,1\right) \right)\\
    &=\left| T \right| ^{\lambda} \left( T \right) \left| T \right|^{-\lambda}.
\end{align*}
  
As $T$ is generalized hyperbolic, then, by using the previous equation and Lemma \ref{main lemma}, it follows that $\Delta _{\lambda}\left( T \right)$, $\lambda \in \left( 0,1 \right)$, is generalized hyperbolic.\\
\indent Conversely, let us suppose that the $\lambda$-Aluthge transform $\Delta_{\lambda} \left( T \right)$ of bounded linear operator $T$ on Hilbert space $\mathbb{H}$ is generalized hyperbolic for every $\lambda \in \left( 0,1 \right)$ then $\Delta_{\lambda} \left( T \right)$ is now invertible. That means $0 \notin \sigma \left( \Delta_{\lambda} \left( T \right) \right)$. Since $\sigma \left( T \right) = \sigma \left( \Delta _{\lambda} \left( T \right) \right)$, $0 \notin \sigma \left( T \right)$ and $T$ turns out to be invertible. Consequently,
  
\begin{align*}
    T &= U \left| T \right| = \left| T \right|^{-\lambda} \left( \left| T \right| ^{\lambda} U \left| T \right|^{1-\lambda}\right) \left| T \right|^{\lambda}\\
    &= \left| T \right|^{-\lambda} \left( \Delta_{\lambda}\left( T \right) \right) \left| T \right|^{\lambda}.
\end{align*}
  
is a conjugated operator with $\Delta_{\lambda} \left( T \right)$. Following Lemma \ref{main lemma}
 again, $T$ is now generalized hyperbolic.
 \end{proof}
As an application of these results, $\lambda$-Aluthge transformations $\Delta_{\lambda} \left( T \right)$ (for every $\lambda \in \left( 0,1 \right)$) preserve the space of shifted hyperbolic operators as follows.

\begin{corollary}
\label{Cor lambda Aluthge of SH}
An automorphism $T$ on Hilbert space $\mathbb{H}$ is shifted hyperbolic if and only if its $\lambda$-Aluthge transform is shifted hyperbolic for every $\lambda \in \left( 0,1 \right)$.
\end{corollary}

\begin{proof}
In Theorem 2 \cite{Cirilo1}, P. Cirilo et al. showed that if $L$ is a generalized hyperbolic operator then $L$ is either hyperbolic or shifted hyperbolic.\\
\indent Let us take an arbitrary shifted hyperbolic operator $T \in B \left( \mathbb{H} \right)$ in which $\mathbb{H}$ is a Hilbert space. Since $T$ is not hyperbolic, $\sigma \left( T \right)$ intersects with the unit circle $\mathbb{T}$ of complex plane $\mathbb{C}$. Moreover, the equality $\sigma \left( T \right) = \sigma \left( \Delta _{\lambda}\left( T \right) \right)$ (for arbitrary $ \lambda \in \left( 0,1 \right)$) provides that $\sigma \left( \Delta _{\lambda} \left( T \right) \right)$ intersects with $\mathbb{T}$ as well. With this, the generalized hyperbolicity of $\Delta _{\lambda} \left( T \right)$ is followed from Theorem \ref{proposition lambda Aluthge transform of GH}.\\
\indent To sum up briefly, $\Delta _{\lambda} \left( T \right)$ is a generalized hyperbolic operator with $\sigma \left( \Delta_{\lambda} \left( T \right) \right) \cap \mathbb{T} \neq \varnothing$ where $\mathbb{T}$ is the unit circle of $\mathbb{C}$. In other words, $\Delta_{\lambda} \left( T \right)$ is a shifted hyperbolic operator for every $\lambda \in \left( 0,1 \right)$.
\end{proof}

Moreover, if $T$ is generalized hyperbolic (w.r.t shifted hyperbolic) then $\Delta _{\lambda} ^{(n)} \left( T \right)$ is generalized hyperbolic (w.r.t shifted hyperbolic) for every $\lambda \in \left( 0,1 \right)$ and $n \in \mathbb{N}$ as shown below.

\begin{corollary}
\label{cor Aluthge iteratess of GH}
For arbitrary $0 < \lambda <1$ and $n \in \mathbb{Z}^{+}$, the $n$-th $\lambda$-Aluthge transform $\Delta_{\lambda} ^{(n)} \left( T \right)$ of generalized hyperbolic operator $T \in B \left( \mathbb{H} \right)$ is generalized hyperbolic where $\mathbb{H}$ is a Hilbert space.
\end{corollary}

\begin{proof}
 
For a generalized hyperbolic operator $T \in B \left( \mathbb{H} \right)$ on Hilbert space $\mathbb{H}$ and any $\lambda \in \left( 0,1 \right)$, the $\lambda$-Aluthge transform $\Delta_{\lambda} \left( T \right)$ is generalized hyperbolic by Theorem \ref{proposition lambda Aluthge transform of GH}. Applying the definition of Aluthge iterates, $\Delta_{\lambda} ^{(2)} \left( T \right) = \Delta_{\lambda} \left( \Delta_{\lambda} \left( T \right) \right)$ is generalized hyperbolic for arbitrary number $\lambda \in \left( 0,1 \right)$. Repeating this process for every $n \in \mathbb{N}$, $\Delta_{\lambda} ^{(n)} \left( T \right) = \Delta_{\lambda} \left( \Delta_{\lambda} ^{(n-1)} \left( T \right) \right)$ are generalized hyperbolic operators.
\end{proof}
As shown in the proof of Theorem \ref{proposition lambda Aluthge transform of GH} ,$T$ and $\Delta_{\lambda} \left( T \right)$ (for every $ \lambda \in \left(0,1\right)$) are conjugate. Hence, they share a variety of common dynamical systems behaviors including topological transitivity, chaotic behaviors and so on.\\
\indent It is well-known that both shadowing and generalized hyperbolicity are open with respect to the norm topology \cite{Cirilo1}. It is a consensus that the hyperbolic property is open too. In the following lemma, we add proof of the hyperbolic case for the sake of completeness.

\begin{lemma}
\label{pepero}
For every invertible hyperbolic operator $T:X \longrightarrow X$ on Banach space $X$, there is $\epsilon>0$ such that every linear operator $S:X \longrightarrow X$ with $\|T-S\|<\epsilon$ is hyperbolic.
\end{lemma}

\begin{proof}
We prove it by a contradiction.\\
Let us suppose that such $\epsilon$ does not exist then there is a sequence of linear operators $S_n \longrightarrow T$ w.r.t norm topology such that $S_n$s are not hyperbolic for every $n \in \mathbb{N}$. Since $T$ is hyperbolic, $\sigma \left( T \right) \cap \mathbb{T} = \varnothing$ in which $\mathbb{T} $ is the unit circle in $\mathbb{C}$. Additionally, $\sigma \left( S_n \right) \cap \mathbb{T} \neq \varnothing$ (for every $n \in \mathbb{N}$). According to the upper semicontinuity of spectrum (see Problem 103 \cite{Halmos1}), $\sigma \left( S_n \right) \subseteq \sigma \left( T \right)$ which is a contradiction.
\end{proof}

Our third result provides a consequence of the previous lemma.

\begin{proof}[Proof of Theorem \ref{prop L is hyperbolic then T is hyperbolic}]
We suppose that the Aluthge iterates $\Delta^{(n)}\left( T \right)$ of a bounded linear operator $T$ on Hilbert space $\mathbb{H}$ converges to a hyperbolic operator $L \in B \left( \mathbb{H} \right)$ under norm topology. Since $L$ is hyperbolic and $\Delta ^{(n)} \left( T \right) \longrightarrow L$, there is $\epsilon>0$ and $N \in \mathbb{N}$ such that
  
\begin{align*}
    \left \| \Delta^{(n)} \left( T \right) -L \right \| < \epsilon, \quad n \in \mathbb{N}, n>N.
\end{align*}
  
By Lemma \ref{pepero}, $\Delta^{(n)} \left( T \right)$ is hyperbolic for each $n > N$. Furthermore, for every $n \in \mathbb{N}$,
\begin{align*}
    \sigma \left(\Delta \left( T \right) \right) = \sigma \left( \Delta ^{(1)} \left( T \right)\right) = \dots = \sigma \left( \Delta^{(n)}\left( T \right) \right)= \dots
\end{align*}
follows that $ \sigma \left( \Delta \left( T \right)\right) \cap \mathbb{T} = \varnothing$. Thus, $T$ is a hyperbolic operator.
\end{proof}
In linear dynamics, weighted shifts form an extremely important class of operators at our disposal to illustrate definitions and to provide examples and counterexamples. Many dynamical properties, including chaotic properties such as Devaney and Li-Yorke chaos, frequently hypercyclicity, topological mixing, as well as hyperbolic properties such as shadowing and expansivity, are analyzed and characterized for weighted shifts even before the general theory of the property under consideration reveal itself \cite{Bernades1, Bernades2, Gilmore1}. For this reason, the class of weighted shifts is also known to be a good model for understanding the dynamics of more complex operators, like, for instance, composition operators \cite{ddm2}. We now turn to investigate bilateral weighted shifts equipped with generalized hyperbolicity.\\

Let us consider $\mathbb{H} = l^2 \left( \mathbb{Z} \right)$. For any bounded sequence of complex numbers $\alpha = \left\{ \alpha _n \right \} _{n \in \mathbb{Z}}$,  the {\em weighted shift} with respect to $\alpha$ is
the linear map $W_\alpha:l^2(\mathbb{Z})\to l^2(\mathbb{Z})$ defined by
  
\begin{align*}
   W_{\alpha} e_n = \alpha_n e_{n+1},
\end{align*}
  
in which $\left \{ e_n \right \} _{n \in \mathbb{Z}}$ is the canonical orthonormal basis of $l^2 \left( \mathbb{Z} \right)$.\\
\indent The matrix representing $W_{\alpha}$ is an $\infty \times \infty$ matrix with the first lower diagonal entries are $\alpha_n$s, i.e
  
\begin{align*}
    T= W_{\alpha} = \begin{bmatrix}
    \dots &\vdots &\vdots &\vdots &\vdots &\vdots &\dots\\
    \dots  &0 &0 &0 &0 &0 &\dots \\
    \dots &\alpha_{-1} &0 &0 &0 &0 &\dots \\
    \dots &0 &\left[ \alpha_{0}\right] &0 &0 &0 &\dots\\
    \dots &0 &0 &\alpha_1 &0 &0 &\dots\\
    \dots &\vdots &\vdots &\vdots &\vdots &\vdots &\dots
    \end{bmatrix}.
\end{align*}
  
The spectrum of $T = W_{\alpha}$ is the annulus
  
\begin{align*}
    \sigma \left( T \right) = \left \{ \lambda \in \mathbb{C} \mid \lambda _1 \le \left| \lambda \right| \le \lambda _2   \right \}
\end{align*}
  
where $\lambda _1 = \liminf _{n \le 0, n \in \mathbb{Z}} \alpha _n$ and $\lambda _2 = \limsup _{n \ge 0, n \in \mathbb{Z}} \alpha _n$.
We also use the following result, which is stated in Ch\={o} et al. \cite{Cho1}.

\begin{lemma}\cite[Theorem 3.4]{Cho1}
\label{hyponormal weighted shift diverges}
Let $T \equiv W_{\alpha}$ be a hyponormal bilateral weighted shift on $l^2 \left( \mathbb{Z} \right)$ with a weight sequence $\alpha \equiv \left \{ \alpha _n \right \} _{n \in \mathbb{Z}}$. Let $a = \inf \left \{ \alpha _n \right \} _{n \in \mathbb{Z}}$ and $b = \sup \left \{ \alpha _n \right \} _{n \in \mathbb{Z}}$. Then $\left \{ \Delta ^{(n)}\left( T \right) \right \}_{n=1} ^{\infty}$ converges to a quasinormal operator in the norm topology if and only if $a=b$.
\end{lemma}

Now, we can prove our last result.

\begin{proof}[Proof of Theorem \ref{Aluthge iterates of SH weighted shifts is divergent}]
We assume that $T = shift \left( \dots, \alpha_{-2}, \alpha _{-1}, \left[ \alpha _0 \right], \alpha_1, \alpha _2, \dots  \right)$ is a shifted hyperbolic bilateral weighted shift on $l^2 \left( \mathbb{Z} \right)$. Let $\left \{ e_n \right \}_{n \in \mathbb{Z}}$ be the canonical orthonormal basis of $l^2 \left( \mathbb{Z} \right)$ then
\begin{equation*}
    T e_n = \alpha e_{n+1}, \quad \forall n \in \mathbb{Z}.
\end{equation*}
The spectrum of each $\Delta ^{(n)} \left( T \right)$ is annulus
  
\begin{align*}
    \sigma \left( \Delta ^{(n)} \left( T \right) \right) = \left \{ \lambda ^{(n)} \in \mathbb{C} \mid \left| \lambda ^{(n)}_1 \right| \le \left| \lambda ^{(n)}  \right| \le \left| \lambda_2 ^{(n)}\right| \right \},
\end{align*}
  
in which $\left| \lambda_1 ^{(n)} \right| = \liminf _{n \in \mathbb{Z}} \alpha_i^{(n)}$ and $\left| \lambda_2 \right| = \limsup_{n\in \mathbb{Z}} \alpha_i ^{(n)}$. Furthermore, the intersection of each $\sigma \left( \Delta ^{(n)} \left( T \right) \right)$ and the unit circle $\mathbb{T}$ of $\mathbb{C}$ follows that $\left| \lambda_1 ^{(n)} \right| <1$ and $\left| \lambda_2 ^{(n)} \right|>1 $ for every $n \in \mathbb{N}$.

Let us suppose that $\Delta ^{(n)} \left ( T \right)$ converges to a quasinormal bilateral weighted shift $L = shift \left( \dots ,\alpha, \left[ \alpha \right],\alpha, \dots, \right)$. By Lemma \ref{hyponormal weighted shift diverges},
the spectrum of $L$ is the circle
  
\begin{align*}
    \sigma \left( L \right) =\left \{ \lambda ' \in \mathbb{C} \mid \left| \lambda' \right| = \alpha \right \}.
\end{align*}
  
\indent On the other hand, the semicontinuity of spectrum leads to
  
\begin{align*}
    \sigma \left( \Delta ^{(n)} \left( T \right) \right) \subseteq \sigma \left( L \right)
\end{align*}
  
that is impossible. Finally, we completed the proof.
\end{proof}

We also prove that the converse of Theorem \ref{prop L is hyperbolic then T is hyperbolic} may not hold. More precisely, even if $T$ is a hyperbolic operator, $\Delta ^{(n)} \left( T \right)$ may be divergent as shown in the following theorem.

We recall the following definition.
\begin{definition}
\label{def hyponormal operator} \cite[Definition 1.22]{IBJ2}
Let $T \in B \left( \mathbb{H} \right)$ in which $\mathbb{H}$ is a Hilbert space. $T$ is said to be hyponormal if $T^*T \ge TT^*$.
\end{definition}

\begin{proof}[Proof of Theorem \ref{Theorem Aluthge iterates of hyperbolic operator}]
Let $X = l^2 \left( \mathbb{Z} \right)$ and $T = shift \left( \dots, 2,2,\left[ 2 \right] , 3, 3, \dots \right)$ be the weighted bilateral shift  on $X$ w.r.t to the bounded sequence of positive numbers $\alpha = \left \{ \alpha _n \right \} _{n \in \mathbb{Z}}$ where
  
\begin{align*}
    \alpha_n = \left \{ \begin{matrix} 2 \quad \text{if } n \le 0 \\
    3 \quad \text{if } n >0 \end{matrix} \right..
\end{align*}
  
\indent It can be observed that $\alpha _n \le \alpha _{n+1}$, for every $ n \in \mathbb{Z}$. Therefore, $T$ is a hyponormal bilateral shift (ref. \cite[Paragraph 1, Part 2]{Ham1}) with $\sup \alpha_n \neq \inf \alpha_n$ . By Lemma \ref{hyponormal weighted shift diverges}, the
iterates $\Delta ^{(n)} \left( T \right)$ diverges under the norm topology. In addition, from
  
\begin{align*}
    \sigma \left( T \right) = \left \{ \lambda \in \mathbb{C} \mid 2 \le \left| \lambda \right| \le 3 \right \},
\end{align*}
  
it follows that $T$ is hyperbolic because $\sigma \left( T \right) \cap \mathbb{T} = \varnothing$, in which $\mathbb{T}$ is the unit circle of complex plane $\mathbb{C}$.
\end{proof}
\section*{Declarations}
The authors did not receive support from any organization for the submitted work.
The author certifies that she has no affiliations with or involvement in any organization or entity with any financial interest or non-financial interest in the subject matter or materials discussed in this manuscript.  \\

\subsection*{Acknowledgment}
The author would like to express her gratitude to Professor Sang Hoon Lee, who is working at Department of Mathematics, Chungnam National University (CNU), for his kind support and encouragement during her Doctoral course at CNU.

\end{document}